\documentclass[11pt]{amsart}

\usepackage{graphics,graphicx}
\usepackage{tikz-cd}
\usetikzlibrary{calc}

\usepackage{dynkin-diagrams}

\usepackage{xcolor}

\usepackage{enumitem}

\usepackage{booktabs}

\usepackage[a4paper]{geometry}
\usepackage[utf8]{inputenc}
\usepackage[T1]{fontenc}

\usepackage{url}

\usepackage{amssymb,amsmath,amsthm}

\usepackage[backref=page,colorlinks=true,allcolors=blue]{hyperref} %

\newtheorem{thm}{Theorem}[section]
\newtheorem{cor}[thm]{Corollary}
\newtheorem{lem}[thm]{Lemma}

\theoremstyle{definition}

\DeclareMathOperator{\SL}{SL}
\DeclareMathOperator{\SO}{SO}
\DeclareMathOperator{\PSO}{PSO}
\DeclareMathOperator{\GO}{O}
\DeclareMathOperator{\GL}{GL}
\DeclareMathOperator{\Sp}{Sp}
\DeclareMathOperator{\PSL}{PSL}
\DeclareMathOperator{\PGL}{PGL}
\DeclareMathOperator{\Aut}{Aut}

\DeclareMathOperator{\grass}{Grass}

\newcommand{\bbC}{\mathbb{C}}
\newcommand{\bbQ}{\mathbb{Q}}
\newcommand{\bbZ}{\mathbb{Z}}
\newcommand{\bbP}{\mathbb{P}}
\newcommand{\bbR}{\mathbb{R}}

\newcommand{\varSLG}{\mathbb{X}_1}
\newcommand{\varF}{\mathbb{X}_2}

\begin{document}

\title[YTD for cohomogeneity one manifolds]{The Yau-Tian-Donaldson conjecture \\ for cohomogeneity one manifolds}

\author{Thibaut Delcroix}
\address{Thibaut Delcroix, Univ Montpellier, CNRS, Montpellier, France}
\email{thibaut.delcroix@umontpellier.fr}
\urladdr{http://delcroix.perso.math.cnrs.fr/}

\date{2020}

\begin{abstract}
We prove the Yau-Tian-Donaldson conjecture for cohomogeneity one manifolds, that is, for projective manifolds equipped with a holomorphic action of a compact Lie group with at least one real hypersurface orbit. 
Contrary to what seems to be a popular belief, such manifolds do not admit extremal Kähler metrics in all Kähler classes in general. %\cite{AndrewHwang} \cite{ChenChengII} 
More generally, we prove that for rank one polarized spherical varieties, \(G\)-uniform K-stability is equivalent to K-stability with respect to special \(G\)-equivariant test configurations. 
This is furthermore encoded by a single combinatorial condition, checkable in practice. 
We illustrate on examples and answer along the way a question of Kanemitsu.
\end{abstract}

\maketitle

\section{Introduction}

A compact complex manifold \(X\) equipped with a holomorphic action of a real compact Lie group \(K\) such that there is at least one real hypersurface orbit \(K\cdot x\) in \(X\) is called a (compact) cohomogeneity one manifold. 
Such manifolds have played a key role in complex geometry, especially in Kähler geometry, for being the easiest non-homogeneous manifolds to study. 
Indeed, under the previous assumption, the generic orbit of \(K\) is a real hypersurface as well, so that any \(K\)-equivariant condition on the manifold must reduce to a one-variable condition. 
It is the underlying reason why Calabi's construction \cite{Cal82} of extremal Kähler metrics on Hirzebruch surfaces works, a construction which gave birth to the Calabi ansatz which applies in many more situations. 
It was also the method Koiso and Sakane \cite{KS86} used to produce the first examples of non-homogeneous compact Kähler-Einstein manifolds with positive curvature. 

Though these initial examples are a bit old, very recent works highlight how these manifolds are still very useful in complex geometry. 
Let us simply illustrate this with one example, where cohomogeneity one manifolds appear under the guise of \emph{two-orbits manifolds} (under a complex Lie group). 
A well studied and long-standing conjecture attributed to Iskovskikh stated that Picard rank one projective manifolds should have a semistable tangent bundle (in the sense of Mumford-Takemoto). 
Kanemitsu \cite{Kan} disproved this conjecture by studying the Picard rank one, two-orbits manifolds whose classification was obtained by Pasquier \cite{Pas09}. 
It was actually not the first appearance of these manifolds in Kähler geometry to disprove a conjecture since the author proved in \cite{DelKSSV} that they provided infinitely many counterexamples to a shorter-lived conjecture of Odaka and Okada \cite{OO13} stating that all Picard rank one Fano manifolds should be K-semistable. We must inform the reader here that the conjecture of Odaka and Okada was disproved as well by Fujita \cite{Fuj17} with two counterexamples. 

In the present note, we will not disprove any conjecture but confirm the Yau-Tian-Donaldson conjecture for projective cohomogeneity one manifolds. 

\begin{thm}
\label{thm_YTD}
On a projective cohomogeneity one manifold, a Kähler class admits a constant scalar curvature Kähler metric if and only if it is K-stable with respect to special equivariant test configurations. 
The latter amounts to a single combinatorial condition checkable in practice. 
\end{thm}

The content of the note is as follows. 
In Section~\ref{sec_recollections} we explain how projective cohomogeneity one manifolds coincide with (non-singular) rank one spherical varieties, briefly recall their classification, then recall some of the results in \cite{KSSV2} for the special case of rank one spherical varieties. 
Section~\ref{sec_proof} is devoted to the proof of our main theorem, and of the corresponding K-stability statement which holds for singular varieties as well. 
In the remaining section, we illustrate the result on some examples. 
It appears that, due to various different hypotheses in papers dealing with cohomogeneity one manifolds, a common belief is that they admit extremal Kähler metrics in all Kähler classes (see e.g. \cite{AndrewHwang,ChenChengII}. 
We thus exhibit an example of cohomogeneity one projective manifold which admits both Kähler classes with cscK metrics and Kähler classes with no extremal Kähler metrics. 
We then answer a question of Kanemitsu on the existence of Kähler-Einstein metrics on non-horospherical Picard rank one manifolds, then study two related Picard rank two cohomogeneity one manifolds and show that they are strong Calabi dream manifolds in the sense of \cite{JesusMartinezGarcia}. 

\subsection*{Acknowledgements}
This research received partial funding from ANR ProjectFIBALGA ANR-18-CE40-0003-01.

\section{Recollections}
\label{sec_recollections}

\subsection{Cohomogeneity one manifolds and spherical varieties}

Let \(X\) be a projective complex manifold, equipped with a holomorphic action of a real connected compact Lie group \(K\). 
The manifold \(X\) is of \emph{cohomogeneity one} if there exists at least one orbit of \(K\) of real codimension one. 
It then follows from \cite[Section~2]{HS82} that the manifold is almost-homogeneous under the action of the complexification \(G:=K^{\bbC}\), that is, \(G\) acts with an open orbit \(G/H\subset X\). 
Furthermore, the complement \(X\setminus G/H\) consists of one or two orbits, on which \(K\) acts transitively. 

If there are two orbits in the complement, they are disconnected and the manifold is \(G\)-equivariantly birational to a \(G\)-homogeneous \(\bbP^1\)-bundle over a generalized flag manifold for \(G\) \cite[Proposition~3.1]{HS82}. 
Such manifolds are, from a different point of view, called rank one horospherical varieties. 
They belong to the large class of spherical varieties, well-studied from the algebraic point of view. 
In fact, in the case where there is only one orbit in the complement, the manifold is a rank one spherical manifold as well \cite[Corollary~2.4]{Cup03}.
We now turn to rank one spherical varieties in general. 

\subsection{On rank one spherical varieties}

Let \(G\) denote a connected complex reductive group, and fix a Borel subgroup \(B\) of \(G\) and a maximal torus \(T\) in \(B\). 
We let \(\mathfrak{X}(T)\) denote the lattice of characters of \(T\). 
We denote by \(R^+\) the set of positive roots of \(G\) and by \(2\varpi=\sum_{\alpha\in R^+}\alpha\) the sum of its positive roots. 

A spherical subgroup of \(G\) is a subgroup \(H\) such that \(BH\) is open in \(G\). 
The homogeneous space \(G/H\) is then called spherical. 
The rank of a spherical homogeneous space is the rank of its weight lattice \(M\), defined as the set of weights of \(B\)-semi-invariant rational functions on \(G/H\). 

Spherical homogeneous spaces of rank one are completely classified up to parabolic induction, by the work of Akhiezer \cite{Akh}. 
More precisely, there is a finite list of families of primitive cases \((G,H)\) (we only list the groups \(G\) and \(H\) up to isogeny, there can be one or two corresponding homogeneous spaces \(G/H\) depending on the couple \((G,H)\)): 
\begin{itemize}
\item the reductive symmetric spaces of rank one 
\begin{itemize}
\item \((\SO_{m+1},S(\GO_1\times\GO_m))\) for \(m\geq 1\),
\item \((\SL_{m+1},S(\GL_1\times\GL_m))\) for \(m\geq 2\),
\item \((\Sp_{2m},\Sp_2\times\Sp_{2n-2})\) for \(n\geq 3\),
\item \((F_4,\SO_9)\),
\end{itemize}   
\item four other affine homogeneous spaces corresponding to the couples \((G_2,\SL_3)\), \((\SO_7,G_2)\), 
\item and three non affine families described precisely in \cite[p. 68, Examples]{Akh}. 
\end{itemize} 
An arbitrary spherical homogeneous space of rank one \(G/H\) is then obtained from the primitive cases as follows. 
There exists a parabolic subgroup \(P\) of \(G\), and a reductive quotient \(\tilde{G}\) of \(P\) such that \(G/H\) is the quotient of \(G\times \tilde{G}/\tilde{H}\) where \(P\) acts diagonally and \(\tilde{G}/\tilde{H}\) is in the list of primitive spherical homogeneous spaces of rank one. 

As already highlighted in the discussion of cohomogeneity one manifolds, there are two very different types of rank one spherical homogeneous spaces, according to whether \(\tilde{G}=\bbC^*\) or \(\tilde{G}\) is semisimple. 
In the case when \(\tilde{G}=\bbC^*\), the resulting homogeneous space \(G/H\) is called a rank one horospherical homogeneous space. 
Then the group of \(G\)-equivariant automorphisms of \(G/H\) is of dimension one. 
Otherwise, the group of \(G\)-equivariant automorphisms of \(G/H\) is finite. 

Spherical varieties of rank one are the normal \(G\)-equivariant embeddings of rank one spherical homogeneous spaces \(G/H\). 
They are classified by colored fans \cite{Kno91} in \(N\otimes\bbR\), where \(N=\operatorname{Hom}(M,\bbZ)\). 
For primitive rank one spherical homogeneous spaces, there is a unique \(G\)-equivariant projective embedding, described in details as well (for most cases) in \cite{Akh}. 
The corresponding fan is without colors and consists either of the toric fan of \(\bbP^1\) (if \(G=\bbC^*\)), or consists of a single one-dimensional cone and its zero dimensional face \(\{0\}\). 
If \(X\) is a projective horospherical rank one spherical variety, its colored fan consists again of the toric fan of \(\bbP^1\), but now each of the two one-dimensional cones can carry colors (and if they do, the corresponding added \(G\)-orbits are not of codimension one). 
If \(X\) is a projective non-horospherical rank one spherical variety, the colored fan again consists of \(\{0\}\) and a single one-dimensional cone, which can now be colored (again, in this case, the added \(G\)-orbit is not of codimension one).  
 If \(X\) is a non-horospherical projective rank one spherical variety,
the generator of \(M\) which evaluates negatively on the colored cone is called the \emph{spherical root} of \(G/H\). 

\subsection{On K-stability of spherical varieties}

Let \(X\) be a rank one \(G\)-spherical variety, with spherical lattice \(M\). 
Let \(N = \operatorname{Hom}(M,\bbZ)\). 
If it is not horospherical, denote by \(\sigma\) its spherical root, and let \(\sigma^*\in N\) be the dual element. 
If it is horospherical, choose any generator \(\sigma\) of \(N\), and let again \(\sigma^*\in N\) be the dual element.  

Let \(L\) be an ample line bundle on \(X\), with moment polytope \(\Delta_+\). 
Recall that the moment polytope is the closure of the set of all \(\lambda/k \in \mathfrak{X}(T)\otimes\bbR\), where \(\lambda\) run over the weights of \(B\)-stable lines in \(H^0(X,L^{\otimes k})\).
It lies in an affine line in \(\mathfrak{X}(T)\otimes \bbR\), with direction \(M\otimes\bbR\). 
Choose an element \(\chi\) of \(\Delta_+\), then there exists \(s_-<s_+ \in \bbQ\) such that
\[\Delta^+ = \{\chi+t\sigma \mid t\in [s_-,s_+]\}. \]
Let \(R_X^+\) denote the set of positive roots of \(G\) which do not vanish identically on \(\Delta_+\). 
Set for \(t\in \bbR\),  
\[ P(t) := \prod_{\alpha\in R_X^+} \frac{\langle \alpha, \chi +t\sigma \rangle}{\langle \alpha, \varpi \rangle}, 
\qquad \qquad 
Q(t) := \left(\sum_{\alpha\in R_X^+} \frac{\langle \alpha, \varpi \rangle}{\langle \alpha, \chi +t\sigma \rangle} \right)P(t). \]
Finally, for any continuous function \(g\) on \([s_-,s_+]\), set 
\[ \mathcal{L}(g) = g(s_-)P(s_-) + g(s_+)P(s_+) - \int_{s_-}^{s_+} 2g(t)(aP(t)-Q(t))\mathop{dt} \]
where \(a\) is such that \(\mathcal{L}(1)=0\), 
and 
\[ \mathcal{J}(g) = \int_{s_-}^{s_+} (g(t)-\inf g)P(t)\mathop{dt}. \]
Note that the moment polytope lies in the positive Weyl chamber of \(G\), that is, all positive roots evaluate non-negatively on elements of \(\Delta_+\). As a consequence, \(P\) and \(Q\) are positive on \(]s_-,s_+[\), and \( \mathcal{J}(g) =0\) if and only if \(g\equiv 0\).

The following criterions for \(G\)-uniform K-stability of \(L\) follow from \cite{KSSV2}. 
Note that we switch here from concave to convex functions to simplify notations. 

\begin{thm}
\label{thm_KSSV2}
\noindent
\begin{enumerate}
\item A polarized rank one \emph{horospherical} variety \((X,L)\) is \(G\)-uniformly K-stable if and only if there exists \(\varepsilon >0\) such that
\[ 
\mathcal{L}(g)  \geq \varepsilon \inf_{l\in \bbR^*} \mathcal{J}(g+l) 
\]
for all rational piecewise linear convex functions \(g:[s_-,s_+]\to\bbR\). 
\item A polarized rank one spherical variety \((X,L)\) which is \emph{not horospherical} is \(G\)-uniformly K-stable if and only if there exists \(\varepsilon >0\) such that
\[ 
\mathcal{L}(g)  \geq \varepsilon \mathcal{J}(g) 
\]
for all non-decreasing rational piecewise linear convex functions \(g:[s_-,s_+]\to\bbR\).  
\end{enumerate}
\end{thm}

\section{Uniform K-stability of rank one spherical varieties}
\label{sec_proof}

In this section we will prove Theorem~\ref{thm_YTD} as a consequence of the following K-stability result applying to singular varieties as well. 

\begin{thm}
\label{thm_main}
A polarized rank \(1\) \(G\)-spherical variety is \(G\)-uniformly K-stable if and only if it is K-stable with respect to \(G\)-equivariant special test configurations. 
\end{thm}

Let us first show how it proves the Yau-Tian-Donaldson conjecture for cohomogeneity one manifolds. 

\begin{proof}[Proof of Theorem~\ref{thm_YTD}]
It suffices to work on rational Kähler classes since the extremal cone is open in the Kähler cone \cite{LS93}, 
and the Kähler cone coincides with the cone of ample real line bundles on spherical manifolds since these manifolds are Mori dream spaces, %
hence rational Kähler classes are dense in the Kähler cone. 
One of the direction is known: existence of cscK metrics implies K-(poly)stability \cite{BDL20}, 
hence in particular K-stability with respect to special equivariant test configurations.
For the other direction, it suffices to apply Theorem~\ref{thm_main} together with Odaka's appendix to \cite{KSSV2}, which shows that for spherical manifolds, \(G\)-uniform K-stability implies the existence of cscK metrics. 
\end{proof}

The result is of course more precise in view of Theorem~\ref{thm_KSSV2}. 
It shows first that for rank one \(G\)-horospherical varieties, \(G\)-uniform K-stability is equivalent to the vanishing of the Futaki invariant on the (at most one-dimensional) center of the group of automorphism. 
Second, if the variety is not horospherical, it admits a unique \(G\)-equivariant special test configuration, and it suffices to check that its Donaldson-Futaki invariant is positive. 

In the course of the proof, we will use the following remarkable properties for the sign of \(aP-Q\). 
Recall first that by definition of a moment polytope, the polynomials \(P\) and \(Q\) are positive on \(]s_-,s_+[\). 

\begin{lem}
\label{lem_when0}
Assume that \(P(s_{\pm})=0\), then \((aP-Q)(t)\) is negative for \(t\in [s_-,s_+]\) close to \(s_{\pm}\).  
\end{lem}

\begin{proof}
Let \(V_{\pm} \subset R_X^+\) be the subset of roots \(\alpha\in R_X^+\) such that \(\langle \alpha, \chi + s_{\pm}\sigma \rangle = 0\). 
If \(V_{\pm}\) is not empty, then the polynomial \(aP\) vanishes to the order exactly \(\operatorname{Card}(V_{\pm})\) at \(s_{\pm}\), while the polynomial \(Q\) vanishes to the order exactly \(\operatorname{Card}(V_{\pm}) -1\) at \(s_{\pm}\). 
It follows that in the same situation, since \(P\) and \(Q\) are positive on \([s_-,s_+]\), \((aP-Q)(t)\) is negative when \(t\in [s_-,s_+]\) is close enough to \(s_{\pm}\). 
\end{proof}

\begin{lem}
\label{lem_nonneglocus}
The locus where \(aP-Q\) is non-negative on \(\Delta\) is \([t_-,t_+]\) for some \(t_{\pm}\in [s_-,s_+]\). 
\end{lem}

\begin{proof}
Since \(P\) is positive on \([s_-,s_+]\), \(aP-Q\) is of the same sign as 
\[a-\sum_{\alpha\in R_X^+}\frac{\langle \alpha, \varpi \rangle}{\langle \alpha, \chi+t\sigma\rangle}\]
on \([s_-,s_+]\). 
Since the inverse of an affine function on \(\bbR\) is convex on the locus where this affine function is positive, the above function is concave on \([s_-,s_+]\). 
It follows that its non-negative locus is a segment in \([s_-,s_+]\). 
\end{proof}

\begin{proof}[Proof of Theorem~\ref{thm_main}]
In order to show the main result by contradiction, we assume that \((X,L)\) is a polarized rank one \(G\)-spherical variety which is K-stable with respect to \(G\)-equivariant special test configurations but not \(G\)-uniformly K-stable. 
For a convex function \(g:[s_-,s_+]\to \bbR\), let us denote by \(\lVert g \rVert\) the quantity 
\(\inf_{l\in\bbR^*}\mathcal{I}(g+l)\) if \(X\) is horospherical, and \(\mathcal{I}(g)\) if not. 
By Theorem~\ref{thm_KSSV2}, since \((X,L)\) is not \(G\)-uniformly K-stable, one can find a sequence \((f_n)\) of rational piecewise linear convex functions from \([s_-,s_+]\) to \(\bbR\) such that for all \(n\), \(\lVert f_n \rVert =1\) and the sequence \((\mathcal{L}(f_n))\) converges to a limit \(l\leq 0\). Note that if the limit can be taken to be strictly negative, then the sequence can be assumed constant. Further note that, if \(X\) is not horospherical, the \((f_n)\) can and are assumed to be non decreasing. 

Let us first modify the sequence a bit. 
Since both \(\mathcal{L}\) and \(\lVert \cdot \rVert\) are invariant under addition of a constant, we may assume that all the functions in the sequence satisfy \(\inf f_n =0\). 
Fix some \(s_0\) in \(]t_-,t_+[\), where \(t_-\) and \(t_+\) are provided by Lemma~\ref{lem_nonneglocus}.
If \(X\) is horospherical, we can further assume that \(0 = \inf f_n = f_n(s_0)\), by adding to \(f_n\) one of its subdifferential at \(s_0\). This does not change \(\lVert f_n \rVert\) by definition, and it does not change \(\mathcal{L}(f_n)\) by the assumption that \((X,L)\) is K-stable with respect to \(G\)-equivariant special test configurations. 
If \(X\) is not horospherical, then since the \(f_n\) are non decreasing, the inf is attained at \(s_-\). 

Under these modifications, the sequence \((\int f_n P )\) is bounded in \(\bbR\). 
If \(X\) is not horospherical, this is immediate since \(\int f_n P = \lVert f_n \rVert = 1\). 
If \(X\) is horospherical, we prove it by contradiction. 
Assume that there is a subsequence of \((f_n)\) (still denoted by \((f_n)\) for simplicity) such that \(\int f_n P \to +\infty\). 
Consider the functions \(g_n=\frac{f_n}{\int f_nP}\). 
Then \(\lVert g_n \rVert \to 0\) while \(\int g_nP = 1\). 
By the pre-compactness result in \cite[Proposition~7.2]{KSSV2}, the sequence \((g_n)\) converges (up to subsequence again) to a function \(g_{\infty}\) on \(]s_-,s_+[\), and the convergence is uniform on all compact subsets. 
The latter ensures that \(\lVert g_n \rVert\) converges to \(\lVert g_{\infty}\rVert\), which is thus equal to zero. 
This is possible only if \(g_{\infty}\) is affine. 
Finally, since \(0 = \inf g_n = g_n(s_0)\), this implies that \(g_{\infty}\) is the zero function. 
This is in contradiction with the convergence \(\lim \int g_n P = 1\).  

Now we can apply the pre-compactness result \cite[Proposition~7.2]{KSSV2} to the sequence \((f_n)\) itself. Up to subsequence, \((f_n)\) thus converges to a convex function \(f_{\infty}\) on \(]s_-,s_+[\) and the convergence is uniform on compact subsets. 

We want to show that, \(\mathcal{L}(f_{\infty})\) is well-defined and less than the limit \(l\) of \(\mathcal{L}(f_n)\).  
Let us first isolate the negative contribution in \(\mathcal{L}(f)\) for an arbitrary non-negative function \(f: [s_-,s_+] \to \bbR\cup\{+\infty\}\) which takes finite values where \(P\) is positive and which is integrable with respect to \((aP(t)-Q(t))\mathop{dt}\). 
By Lemma~\ref{lem_nonneglocus}, there exists \(t_-<t_+\) in \([s_-,s_+]\) such that \(aP-Q\) is non-negative exactly on \([t_-,t_+]\). 
It follows that the negative contribution to \(\mathcal{L}(f)\) is 
\[ - \int_{t_-}^{t_+}2f(t)(aP(t)-Q(t))\mathop{dt} \]
We claim that this is always well defined and finite if for \(f_{\infty}\). 
If \(P\) is strictly positive on \([t_-,t_+]\), then since \(\int f_{\infty} P =1\), the claim holds. 
It is actually always the case that \(P\) is strictly positive on \([t_-,t_+]\), by Lemma~\ref{lem_when0}.

Since we assumed \(\lim \mathcal{L}(f_n) \leq 0\), and the negative contribution in \(\mathcal{L}(f_n)\) converges to \(- \int_{t_-}^{t_+}2f_{\infty}(t)(aP(t)-Q(t))\mathop{dt}\), the positive contribution of \(\mathcal{f_n}\) must converge as well. 
Since the positive contribution is the sum of positive terms 

\[ 
f_n(s_-)P(s_-) + f_n(s_+)P(s_+) - \int_{s_-}^{t_-} 2f_n(t)(aP(t)-Q(t))\mathop{dt} - \int_{t_+}^{s_+} 2f_n(t)(aP(t)-Q(t))\mathop{dt} 
\]
each of these terms must be bounded. 
This implies that \(f_{\infty}\) is integrable with respect to \((aP(t)-Q(t))\mathop{dt}\) and that \(\lim_{t\to s_{\pm}} f_{\infty}(t)\) is finite when \(P(s_{\pm})\) is non-zero. 
By a slight abuse of notations, we let \(f_{\infty}: [s_-,s_+] \to \bbR \cup \{+\infty\}\) be the unique lower semi-continuous extension of \(f_{infty}\). 
The last part of the penultimate sentence shows that \(f_{\infty}(s_{\pm})\) is finite whenever \(P(s_{\pm})\) is, so \(\mathcal{L}(f_{\infty})\) is well defined and \(\mathcal{L}(f_{\infty})\leq l\). 

Consider the affine function 
\[ h(t) = f_{\infty}(t_-) + \frac{t-t_-}{t_+-t_-}(f_{\infty}(t_+)-f_{\infty}(t_-)). \]
Note that the values \(f_{\infty}(t_{\pm})\) are finite by the discussion above. 
Convexity of \(f_{\infty}\) implies \(h \leq f_{\infty}\) on \([s_-,t_-] \cup [t_+,s_+]\) and \(h\geq f_{\infty}\) on \([t_-,t_+]\). 
Thus both the positive and negative contribution in \(\mathcal{L}(h)\) are lower than that in \(\mathcal{L}(f_{\infty})\), hence 
\[0 = \mathcal{L}(h)\leq \mathcal{L}(f_{\infty})\leq l \leq 0.\]
In particular, we have shown that K-stability with respect to \(G\)-equivariant special test configurations implies \(G\)-equivariant K-semistability. 

To conclude the proof it remains to obtain a contradiction with the initial definition of the sequence \((f_n)\). 
This final argument depends on the nature of \(X\). 
If \(X\) is not horospherical then all the functions \(f_n\) are non decreasing, hence \(f_{\infty}\) and \(h\) as well. 
If the slope of \(h\) is strictly positive, then \(\mathcal{L}(h)>0\) by assumption which provides the contradiction. 
Else \(h\) is constant. 
Since \(f_{\infty}\) is non decreasing, \(f_{\infty}\) is constant on \([s_-,t_+]\). 
All \(f_n\) satisfy \(\inf f_n = f_n(s_-) = 0\), so \(f_{\infty}(s_-)=0\). 
But then either \(f_{\infty} \equiv 0\), which contradicts \(\lVert f_{\infty} \rVert = 1\), or \(\mathcal{L}(f_{\infty})>0\), which is another contradiction. %bigger than h only where contribution is positive

We now conclude the horospherical case. 
Assume first that \(f_{\infty}\) is affine.
All the functions \(f_n\) satisfy \(f_n(s_0)=\inf f_n =0\), hence \(f_{\infty}\) as well. Since \(s_-<s_0<s_+\), this shows that \(f_{\infty}\) is the zero function, a contradiction with \(\int f_{\infty}P = 1\). 
Assume now that \(f_{\infty}\) is not affine, hence that \(f_{\infty}\neq h\). 
Since \(aP-Q\) is a non-zero polynomial, by considering the positive and negative contribution in \(\mathcal{L}\) as before, we see that \(\mathcal{L}(f_{\infty})>\mathcal{L}(h)=0\). 
This is the final contradiction. 
\end{proof}

%%%%%%%%%%%%%%%%%%%%%%%%%%%%%%%%%%%%%%%%%%%%%%%
%%%%%%%%%%%%%%%%%%%%%%%%%%%%%%%%%%%%%%%%%%%%%%%

\section{Examples}
\label{sec_examples}

\subsection{An example of Kähler class with no extremal Kähler metrics}

We will here consider an example initially encountered in \cite{KRFSS}. 
There, we considered as an ingredient of the proof the existence of Kähler-Einstein metrics on some blow-down of the \(G_2\)-stable divisors in the wonderful compactification of \(G_2/\SO_4\). 
Such varieties are rank one spherical (horosymmetric) varieties, Fano with Picard rank one, and one of these does not admit (singular) Kähler-Einstein metrics. 
If we go back to the corresponding \(G_2\)-stable divisor in the wonderful compactification of \(G_2/\SO_4\), which is smooth, this should provide an example of cohomogeneity one manifold and Kähler classes on it with no extremal Kähler metrics. 
We verify this in the following paragraphs. 

\subsubsection{Recollection on the group \(G_2\)}

We consider the exceptional group \(G_2\) with a fixed choice of Borel subgroup \(B\) and maximal torus \(T\), and an ordering of simple roots as in Bourbaki's numbering, so that \(\alpha_1\) is the short root and \(\alpha_2\) is the long root. 

Up to scaling, the Weyl group invariant scalar product on \(\mathfrak{X}(T)\) satisfies \(\langle \alpha_1, \alpha_1\rangle = 2\), \(\langle \alpha_1, \alpha_2\rangle = -3\) and \(\langle \alpha_2, \alpha_2\rangle = 6\).
The fundamental weight for \(\alpha_1\) is \(2\alpha_1+\alpha_2\) and the fundamental weight for \(\alpha_2\) is \(3\alpha_1+2\alpha_2\). 
The positive roots and their scalar product with an arbitrary element \(x_1\alpha_1+x_2\alpha_2\) read 
\begin{align*}
\langle \alpha_1, x_1\alpha_1+x_2\alpha_2 \rangle &= 2x_1-3x_2 \\
\langle \alpha_2, x_1\alpha_1+x_2\alpha_2 \rangle &= 3(-x_1+2x_2) \\
\langle \alpha_1+\alpha_2, x_1\alpha_1+x_2\alpha_2 \rangle &= -x_1+3x_2 \\
\langle 2\alpha_1+\alpha_2, x_1\alpha_1+x_2\alpha_2 \rangle &= x_1 \\
\langle 3\alpha_1+\alpha_2, x_1\alpha_1+x_2\alpha_2 \rangle &= 3(x_1-x_2) \\
\langle 3\alpha_1+2\alpha_2, x_1\alpha_1+x_2\alpha_2 \rangle &= 3x_2 
\end{align*}
The half-sum of positive roots is \(\varpi=5\alpha_1+3\alpha_2\). 

\subsubsection{The facet of the wonderful compactification of \(G_2/\SO_4\) and its Kähler classes}

Let \(P_1\) be the parabolic subgroup of \(G_2\) containing \(B\) such that \(-\alpha_1\) is not a root of \(P_1\). 
Its Levi factor has adjoint form \(\PSL_2\). 
Let \(X\) be the (non-singular) horosymmetric variety obtained by parabolic induction from the \(P_1\)-variety \(\bbP^2\), considered as the projectivization of the space of equations of quadrics in \(\bbP^1\) on which \(P_1\) acts via the natural action of \(\PSL_2\) on \(\bbP^1\). 
Its open orbit \(G_2/H\) is the corresponding parabolic induction from \(\PSL_2/\PSO_2\). 
Note that \(X\) is the wonderful compactification of \(G_2/H\) and that since it is a parabolic induction, \(\Aut^0(X)=G_2\) \cite[Proposition~3.4.1]{Pez09}. 

The variety \(X\) is a Picard rank two horosymmetric variety. 
Its spherical root is \(\sigma = 2\alpha_2\), and its spherical lattice \(M\) is the lattice generated by \(\sigma\). 
We can describe its Kähler cone by using \cite{Bri89}, recalled for the special case of horosymmetric varieties in \cite{DelHoro,DH}. 
Any Kähler class on a projective spherical variety is the class of a real divisor. 
The vector space of classes of real divisor is generated by the classes of all prime \(B\)-stable divisors modulo the relations imposed by \(B\)-semi-invariant rational functions. Here the prime \(B\)-stable divisors are the closure \(E\) of the unique codimension one \(G_2\)-orbit (obtained by parabolic induction from the space of degenerate quadrics in \(\bbP^1\), that is, points), and the closures of the two colors \(D_1\) and \(D_2\) in \(G/H\), where \(D_1\) is the only codimension one \(B\)-orbit not stable under \(P_1\) (the codimension one \(P_2\)-orbit obtained by moving the unique color in \(\bbP^2\)), and \(D_2\) is the only codimension one \(B\)-orbit not stable under \(P_2\). 
Note that \(D_1\) is also the pull-back of the ample generator of the Picard group of \(G/P_1\). 
Since the spherical rank of \(G/H\) is one, there is a single relation to consider, which amounts to \(2D_2-E-6D_1=0\), since the image of \(D_1\) by the color map is the restriction \(-6\sigma^*\) of the coroot \(\alpha_1^{\vee}\) to \(M\otimes \bbR\), and the image by the color map of \(D_2\) is the restriction \(2\sigma^*\) of the positive restricted coroot \(\frac{\alpha_2^{\vee}}{2}\) (the image of \(E\) is the primitive generator of the valuation cone \(-\sigma^*\)). 

In view of the above presentation, we can write a real divisor as \(sE+s_1D_1\). 
Since K-stability is invariant under scaling of the Kähler class, we may as well assume \(s_1=6\). 
Brion's ampleness criterion for the real line bundle \(sE+6D_1\) translates simply to the condition \(0<s<1\), and the moment polytope is then 
\[ \Delta_+(s) = 6(2\alpha_1+\alpha_2) + \{2t\alpha_2 \mid 0\leq t\leq s \}. \]

\subsubsection{K-stability condition}

We have 
\[ P(t) = \frac{288}{5}t(1-t^2)(9-t^2) \]
and
\begin{align*} 
Q(t) & = \left(\frac{1}{6(1-t)} + \frac{1}{4t} + \frac{2}{3(1+t)} + \frac{5}{12} + \frac{1}{3-t} + \frac{3}{2(3+t)}\right) P(t) \\
& = \frac{24}{5}(5t^5+15t^4-150t^3-90t^2+225t+27).
\end{align*}
To avoid computing the quantity \(a\) independently, we rewrite the K-stability condition as 
\begin{equation} 
\label{Kstabcondition}
\left(s P(s) + 2\int_0^s tQ(t)\mathop{dt}\right)\int_0^s P(t)\mathop{dt} - 
\left(P(s) + 2\int_0^s Q(t)\mathop{dt}\right)\int_0^s t P(t)\mathop{dt}
> 0 
\end{equation}
The left hand side above is the polynomial 
\[ R(s):=\frac{1152}{175}s^4(11s^8+20s^7-348s^6-240s^5+3123s^4+1260s^3-9072s^2+5103). \]
One can plug in specific values to check that 
\[ R\left(\frac{1}{2}\right)= \frac{7315083}{5600} >0 \] 
and 
\[ R\left(\frac{98}{100}\right)= -\frac{12097691278181901659043}{47683715820312500000} <0. \]
In other words, there are Kähler classes on \(X\) with cscK metrics and Kähler classes with no cscK metrics. 
Since \(\Aut^0(X)=G_2\) is semisimple, a Kähler class with no cscK metrics does not admit any extremal Kähler metric either. 
Using numerical approximation, one can be more precise: the Kähler class \(sE+6D_1\) contains a cscK metric if and only if \(s<s_0\), where \(s_0\simeq 0.97202\). 

\subsection{Strong Calabi dream manifolds of cohomogeneity one, and an answer to a question of Kanemitsu}

We will now provide examples of cohomogeneity one manifolds which are not horospherical and are strong Calabi Dream manifolds in the sense of \cite{JesusMartinezGarcia}. 
We take a small detour and choose slightly complicated manifolds to answer along the way a question of Kanemitsu \cite[Remark~4.1]{Kan} on the existence of Kähler-Einstein metrics on some cohomogeneity one manifolds with Picard rank one. 

By Pasquier's classification \cite{Pas09}, there are two Picard rank one, non-horospherical cohomogeneity one manifolds, one acted upon by \(\PSL_2\times G_2\), that we will denote by \(\varSLG\), and one acted upon by \(F_4\), that we will denote by \(\varF\). 
Both are two orbit varieties with semisimple automorphism group. 
More precisely, \(\Aut^0(\varSLG)=\PSL_2\times G_2\) and \(\Aut^0(\varF)=F_4\). 
We will first prove:

\begin{thm}
There exist Kähler-Einstein metrics on \(\varSLG\) and \(\varF\). 
\end{thm}

As a corollary, we recover a result of \cite{Kan}.

\begin{cor}
The tangent bundles of \(\varSLG\) and \(\varF\) are Mumford-Takemoto stable. 
\end{cor}

These manifolds \(\mathbb{X}_i\) each admit a unique \emph{discoloration} \(\widetilde{\mathbb{X}_i}\) which is a smooth projective Picard rank two cohomogeneity one manifold which surjects equivariantly to \(\mathbb{X}_i\) and where the complement of the open orbit is of codimension one. 
We will apply Theorem~\ref{thm_YTD} to obtain:

\begin{thm}
The manifolds \(\widetilde{\varSLG}\) and  \(\widetilde{\varF}\) admits a cscK metric in all Kähler classes. 
In other words, they are \emph{Calabi dream manifolds} in the terminology of \cite{ChenChengII}, and more precisely \emph{strong Calabi dream manifolds} in the terminology of \cite{JesusMartinezGarcia}.
\end{thm}

We will in the paragraphs to follow provide the combinatorial data associated to the manifolds under study. 
It is rather easy since these are horosymmetric. 
For the discolorations, we will then determine the Kähler classes and compute the K-stability condition as in the previous example. 
For the Kähler-Einstein metrics, it is a bit faster to use directly the criterion in \cite{DelKSSV} since one needs only the polynomial \(P\) up to scalar, and the polynomial \(Q\) is not needed. 

\subsubsection{Kähler-Einstein metrics on \(\varSLG\)}

Let \(G\) denote the group \(\PSL_2\times G_2\). 
We fix a choice of Borel subgroup \(B\) and of maximal torus \(T\subset B\). 
Let \(\alpha_0\) denote the positive root of \(\SL_2\) and let \(\alpha_1\) and \(\alpha_2\) denote the simple roots of \(G_2\), numbered so that \(\alpha_1\) is the short root (in accordance with Bourbaki's standard numbering and with the previous example). 
We can choose a Weyl group invariant scalar product on \(\mathfrak{X}(T)\) satisfying \(\langle \alpha_0, \alpha_0\rangle = 1\) and the same scaling as in the previous example for the restriction to \(G_2\). 
Of course, the root \(\alpha_0\) is orthogonal to \(\alpha_1\) and \(\alpha_2\). 

It follows from the description of the variety \(\varSLG\) in \cite{Pas09} that its open orbit \(G/H\) is obtained by parabolic induction from the rank one symmetric space \(\PSL_2\times\PSL_2/\PSL_2\), where the parabolic subgroup of \(G\) is the parabolic \(P_2\) associated to the long root \(\alpha_2\), whose Levi factor has adjoint form \(\PSL_2\times\PSL_2\). 
The spherical lattice \(M\) for \(\varSLG\) is thus the lattice generated by \(\alpha_0+\alpha_1\). 

Furthermore, the variety \(\varSLG\) is the unique fully colored compactification of \(G/H\). 
It follows that the moment polytope \(\Delta_+\) corresponding to the anticanonical line bundle is the intersection with the positive Weyl chamber of the affine line with direction \(\bbR(\alpha_0+\alpha_1)\) passing through the sum of positive roots \( \alpha_0 + 10 \alpha_1 + 6 \alpha_2 \). 
If we write the moment polytope as 
\[ \Delta_+ = \{ (1+t)\alpha_0+(10+t)\alpha_1)+6\alpha_2 \mid u \leq t \leq v \} \]
then we can determine \(u\) and \(v\) as the extreme values of \(t\) such that 
\(\langle \alpha_i, (1+t)\alpha_0+(10+t)\alpha_1)+6\alpha_2\rangle \geq 0\) for \(i\in\{0,1,2\}\), that is, \(u=-1\) and \(v=2\).  

We may finally compute the K-stability condition, which is 
\[ \int_{-1}^2 t (1+t)^2(2-t)(8-t)(10+t)(4+t)\mathop{dt} >0. \]
By direct computation, the integral is equal to \(\frac{120285}{56}\) hence the condition is satisfied.

\subsubsection{CscK metrics on the discoloration \(\widetilde{\varSLG}\)}

The discoloration \(\widetilde{\varSLG}\) of \(\varSLG\) is obtained by the following parabolic induction procedure. 
Take the quotient of \(\bbP^3 \times G\) by the diagonal action of the parabolic \(P_2\), where the action on \(G\) is obvious and the action on \(\bbP^3\) is induced by the action of the Levi factor of \(P_2\) and the obvious structure of two-orbit \(\PGL_2\times\PGL_2\)-variety on \(\bbP^3\) seen as the projectivization of \(2\) by \(2\) matrices. 
% remark: The closed orbit in \bbP^3 is the projectivization of rank one matrices. 
It is a homogeneous \(\bbP^3\)-bundle over the generalized flag manifold \(G/P_2\).

It is a rank one horosymmetric variety with Picard rank two. 
We can describe its Kähler cone by using \cite{Bri89}, recalled for the special case of horosymmetric varieties in \cite{DelHoro,DH}. 
Any Kähler class on a projective spherical variety is the class of a real divisor. 
The vector space of classes of real divisor is generated by the classes of all prime \(B\)-stable divisors modulo the relations imposed by \(B\)-semi-invariant rational functions. Here the prime \(B\)-stable divisors are the (\(G\)-stable) exceptional divisor \(E\), and the closures of the two colors \(D_{01}\) and \(D_2\) in \(G/H\), where \(D_{01}\) is the only codimension one \(B\)-orbit not stable under \(P_{0}\) and \(P_{1}\), and \(D_2\) is the only codimension one \(B\)-orbit not stable under \(P_2\). 
Note that \(D_2\) is also the pull-back of the ample generator of the Picard group of \(G/P_2\). 
Since the spherical rank of \(G/H\) is one, there is a single relation to consider, which amounts to \(E+D_2-2D_1=0\), since the image of \(D_2\) by the color map is the restriction of the coroot \(\alpha_2^{\vee}\) to \(M\otimes \bbR\) which coincides with the generator of the colorless ray corresponding to \(E\), and the image by the color map of \(D_{01}\) is the only positive restricted coroot, the restriction of \(\frac{1}{2}(\alpha_0^{\vee}+\alpha_1^{\vee})\) to \(M\otimes \bbR\). 

The class of any real divisor is thus represented by a \(s_E E+s_2D_2\) for \(s_E\) and \(s_2\) two real numbers. 
By Brion's ampleness criterion, it is a Kähler class if and only if \(0<s_E<s_2\), and the moment polytope is then 
\[ \Delta_+(s_E,s_2) := s_2(3\alpha_1+2\alpha_2) + \{ t(\alpha_0+\alpha_1)\mid 0 \leq t \leq s_E \}, \]
where \(3\alpha_1+2\alpha_2\) is to be though of as the fundamental weight of \(\alpha_2\) here.  

\begin{figure}
\centering
\caption{Moment polytope \(\Delta_+(s,1)\)}
\label{fig_polSLG}
\begin{tikzpicture}  
\draw [dashed] (0,0) -- (0,3);
\draw [dotted] (0,0) -- (3,0);
\draw [dashed] (0,0) -- (3,3);
\draw [thick] (0,2) -- (1,2);
\draw [dotted] (0,2) -- (2,2);
\draw (0,2) node{$\bullet$};
\draw (0,2) node[below left]{{\scriptsize $3\alpha_1+2\alpha_2$}};
\draw (2,0) node{$\bullet$};
\draw (2,0) node[above]{{\scriptsize $\alpha_0+\alpha_1$}};
\draw (1,0) node{$\bullet$};
\draw (1,0) node[below]{{\scriptsize $s(\alpha_0+\alpha_1)$}};
\draw (1,2) node{$\bullet$};
\end{tikzpicture}
\end{figure}
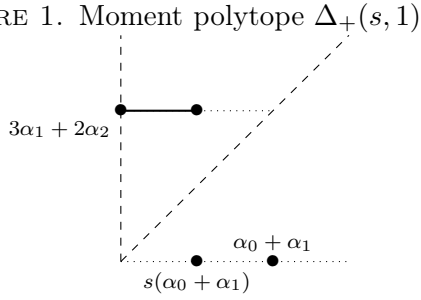

We may now compute the unique condition for K-stability given by an equivariant special test configuration for these polarizations. 
Since a Kähler class is K-stable if and only if one of its positive multiple is, we can assume \(s_2=1\) and write \(s:=s_E\) in the following, to simplify notations. 
Let \(P\) and \(Q\) denote the polynomials 
\begin{align*}
P(t) &= \frac{t^2}{15}\left( t^4 - 10 t^2 + 9 \right) \\
Q(t) &= \frac{t}{30}\left( 3t^5 + 6t^4 - 90t^3 - 40t^2 +135t +18 \right)
\end{align*}
The Kähler class \(s E + D_2\) is \(G\)-uniformly K-stable if and only if \(s\) satisfies the condition~\eqref{Kstabcondition}.
By computing the polynomial on the left hand side, the condition is 
\[ \frac{s^6}{132300}(9s^8+42s^7-266s^6-378s^5+2135s^4+1764s^3-5292s^2+2646) > 0, \]
and one can check that this condition is satisfied for all \(s\in ]0,1[\). 

\subsubsection{Kähler-Einstein metrics on the \(F_4\)-variety \(\varF\)}

Let \(\alpha_i\) denote the simple roots of \(F_4\), ordered in accordance with Bourbaki's numbering.
%, that is, as indicated on the Dynkin diagram \dynkin[label,ordering=Bourbaki]F4. 
Up to scaling, the Weyl group invariant scalar product is such that the matrix of \(\langle \alpha_i,\alpha_j \rangle\) is given by 
\[ 
\begin{pmatrix}
2 & -1 & 0 & 0 \\
-1 & 2 & -1 & 0 \\
0 & -1 & 1 & -\frac{1}{2} \\
0 & 0 & -\frac{1}{2} & 1
\end{pmatrix}
\]

It follows from the description of the variety \(\varF\) in \cite{Pas09}
that its open orbit \(F_4/H\) is obtained by parabolic induction from the rank one symmetric space \(\Sp_6/\Sp_2\times \Sp_4\), where the parabolic subgroup of \(F_4\) is the parabolic \(P_1\) associated to the root \(\alpha_1\). 
The spherical lattice \(M\) for \(\varF\) is the lattice generated by the restricted root of the symmetric space, \( \beta := \alpha_2+2\alpha_3+\alpha_4 \). 

Furthermore, the variety \(\varF\) is the unique fully colored compactification of \(F_4/H\). 
It follows that the moment polytope corresponding to the anticanonical line bundle is the intersection with the positive Weyl chamber of the affine line with direction \(\bbR\beta \) passing through \( 16\alpha_1 + 29\alpha_2 + 42\alpha_3 +21\alpha_4 \), the sum of positive roots minus the sum of positive roots of \(\Sp_6\) fixed by the involution defining the symmetric space.   
More explicitly, the moment polytope is 
\[ \Delta_+ = \{ t\beta+8\omega_1 \mid 0\leq t\leq 8 \} \]
where \(\omega_1 = 2\alpha_1+3\alpha_2+4\alpha_3+2\alpha_4\) is the fundamental weight for \(\alpha_1\). 

The K-stability condition for the anticanonical line bundle thus reads 
\[ \int_0^8 (t-5)t^7(256-t^2)^2(64-t^2)^2\mathop{dt} >0. \]
This condition is satisfied since the left hand side is equal to 
\[ \frac{3672386428957884416}{153153}. \] 

\subsubsection{CscK metrics on the discoloration \(\widetilde{\varF}\)}

The discoloration \(\widetilde{\varF}\) of \(\varF\) is obtained by the following parabolic induction procedure. 
Take the quotient of \(\grass(2,6)\times F_4\) by the diagonal action of the minimal parabolic \(P_1\), where the action on \(F_4\) is obvious and the action on \(\grass(2,6)\) is induced by the action of the Levi factor of \(P_1\) and the structure of two-orbit \(\Sp_6\)-variety on \(\grass(2,6)\) (this is the wonderful compactification of the symmetric space \(\Sp_6/\Sp_2\times\Sp_4\)). 
%remark: the closed \(\Sp_6\)-orbit in \grass(2,6) is the GFM associated to the maximal parabolic P_2 in Bourbaki's numbering. 

It is a rank one horosymmetric variety with Picard rank two. 
Again, its Kähler cone is determined from combinatorial data using \cite{Bri89,DelHoro,DH}. 
Here, the vector space of real divisors is the quotient of the three dimensional vector space generated by the exceptional divisor \(E\) and the closure of two colors \(D_1\) and \(D_3\), where \(D_i\) is the closure of the only codimension one \(B\)-orbit not stable under the minimal parabolic \(P_i\) where \(i\in\{1,3\}\), by the relation \( D_1 + E -2 D_3 = 0 \).
The relation follows from the fact that the image of \(D_1\) under the color map is the restriction of the coroot \(\alpha_1^{\vee}\) to \(M\otimes \bbR\), which coincides with the primitive generator of the colorless ray corresponding to the \(F_4\)-orbit \(E\), and the image by the color map of \(D_3\) is the only positive restricted coroot, equal to (the restriction to \(M\otimes\bbR\) of) \((\alpha_2+2\alpha_3+\alpha_4)^{\vee}\), which coincides with the double of the opposite of the generator of \(E\). 
Note that \(D_1\) is the pull-back of the ample generator of the Picard group of \(F_4/P_1\). 

The class of any real divisor is thus represented, up to multiple, by some \(s E + D_1\). 
By Brion's ampleness criterion, it is a Kähler class if and only if 0<s<1, 
and the moment polytope is then 
\[ \Delta_+ = \{ t\beta + \omega_1 \mid 0 \leq t \leq s \}. \]
The Kähler class \(sE+D_1\) is \(F_4\)-uniformly K-stable if and only if \(s\) satisfies the condition~\eqref{Kstabcondition}
where, here, the polynomials \(P\) and \(Q\) are given by 
\begin{align*}
P(t) & = \frac{1}{2^7}t^7(4-t^2)^2(1-t^2)^2\\
Q(t) & = \frac{1}{2^8}t^6(4-t^2)(1-t^2)(13t^5+22t^4-105t^3-110t^2+116t+88).
\end{align*}
It is a tedious but workable task to verify that the polynomial on the left hand side of condition~\eqref{Kstabcondition} is positive for \(s\in ]0,1[\). 

\bibliographystyle{alpha}
\bibliography{RK1}

\end{document}